\documentclass[12pt,reqno]{amsart}
\newtheorem{thm}{Theorem}[section]
\newtheorem{defi}{Definition}[section]
\newtheorem{lem}{Lemma}[section]

\theoremstyle{definition}
\newtheorem{rem}{Remark}[section]
\newtheorem{exmp}{Example}[section]
\newcommand{\be}{\begin{equation}}
\newcommand{\ee}{\end{equation}}
\newcommand{\bea}{\begin{eqnarray}}
\newcommand{\eea}{\end{eqnarray}}
\newcommand{\beb}{\begin{eqnarray*}}
\newcommand{\eeb}{\end{eqnarray*}}
\newcommand{\norm}[1]{\left\lVert#1\right\rVert}
\usepackage{amssymb,amsfonts,amsthm,setspace,indentfirst,mathrsfs}
\usepackage{minibox}
\usepackage{enumitem}
\usepackage{mathtools}
\usepackage{hyperref}
\numberwithin{equation}{section}

\begin{document}
\title[ON $\mathcal I$ AND $\mathcal I^*$-EQUAL CONVERGENCE IN LINEAR $2$-NORMED SPACES]{ON $\mathcal I$ AND $\mathcal I^*$-EQUAL CONVERGENCE IN LINEAR 2-NORMED SPACES}

\author[A.K.Banerjee]{Amar Kumar Banerjee$^1$ }
\author[N.Hossain]{Nesar Hossain$^2$}

\address{$^{1,2,}$Department of Mathematics, The University of Burdwan, Golapbag, Burdwan - 713104, West Bengal, India.}

\email{$^1$akbanerjee@math.buruniv.ac.in; akbanerjee1971@gmail.com}
\email{$^2$nesarmath94@gmail.com; nesarhossain24@gmail.com}


\maketitle

\begin{abstract}
 In this paper we study the notion of $\mathcal{I}$ and $\mathcal{I^*}$-equal convergence in linear 2-normed spaces and some of their properties. We also establish the relationship between them.
\end{abstract}

\section{\bf{Introduction}}

The idea of usual convergence of a real sequence was extended to statistical convergence independently by Fast \cite{Fast} and Steinhaus \cite{Steinhaus} in the year $1951$. Lot of developments were made on this notion of convergence after the pioneering works of Šalát \cite{Salat} and Fridy \cite{Fridy}. After long fifty years, the concept of statistical convergence was extended to the idea of $\mathcal{I}$-convergence depending on the structure of ideals $\mathcal{I}$ of $\mathbb{N}$, the set of natural numbers, by Kostyrko et al. \cite{Kostyrko Salat Wilczynski}. Throughout the paper $\mathbb{N}$ and $\mathbb{R}$ denote the set of all positive integers and the set of all real numbers respectively. The ideal $\mathcal{I}$ is a collection of subsets of $\mathbb{N}$ which is closed under finite union and with the condition that $B\in \mathcal{I}$ whenever $B\subset A\in \mathcal{I}$. Indeed, the concept of $\mathcal{I}$- convergence of real sequences is a generalization of statistical convergence which is based on the structure of the ideal $\mathcal{I}$ of subsets of $\mathbb{N}$. $\mathcal{I}$-convergence of real sequences coincides with the ordinary convergence if $\mathcal{I}$ is the ideal of all finite subsets of $\mathbb{N}$ and with the statistical convergence if $\mathcal{I}$ is the ideal of $\mathbb{N}$ of natural density zero. In \cite{Kostyrko Salat Wilczynski} the concept of $\mathcal{I^*}$-convergence was also introduced. Last few years several works on $\mathcal{I}$-convergence and its related areas were carried out in different directions in different spaces viz. metric spaces, normed linear spaces, probabilistic metric spaces, $S$-metric spaces, linear $2$-normed spaces, cone metric spaces, topological spaces etc. (see \cite{Banerjee Banerjee 1, Banerjee Banerjee 2, Banerjee Mondal, Lahiri Das} and many more references therein). Ordinary convergence always implies statistical convergence and when $\mathcal{I}$ is admissible ideal, $\mathcal{I^*}$-convergence implies $\mathcal{I}$-convergence. But the reverse implication does not hold in general. But when $\mathcal{I}$ satisfies the condition $(AP)$,  $\mathcal{I}$-convergence implies $\mathcal{I^*}$-convergence. A remarkable observation is that a statistically convergent sequence and $\mathcal{I}$ and $\mathcal{I^*}$-convergent saequence need not even be bounded.

Recently some significant investigations have been done on sequences of real valued functions by  using the idea of statistical and $\mathcal{I}$-convergence \cite{Caserta Maio Kocinac, Das  Dutta  Pal, Komisarski, Mrozek}. The interesting notion of equal convergence was introduced by Császár and Laczkovich \cite{Csaszar Laczkovich} for sequences of real valued functions (also known as quasinormal convergence \cite{Bukovska}). It is known that equal convergence is weaker than uniform convergence and stronger than pointwise converegence for the sequences of real valued functions. A detailed investigation was carried out by Császár and Laczkovich in \cite{Csaszar Laczkovich} on such type of convergence. In \cite{Das Dutta, Das Dutta Pal, Filipow Staniszewski} the concept of equal convergence of sequences of real functions was generalized to the ideas of $\mathcal{I}$ and $\mathcal{I^*}$-equal convergence using ideals of $\mathbb{N}$ and the relationship between them were investigated. $\mathcal{I}$-equal convergence is weaker than $\mathcal{I}$-uniform convergence and stronger than $\mathcal{I}$-pointwise convergence \cite{Das Dutta Pal}. 

The notion of linear $2$-normed spaces was initially introduced by Gahler \cite{Gahler} and since then the concept has been studied by many authors. In \cite{Sarabadan Talebi} some significant investigations on $\mathcal{I}$-uniform and $\mathcal{I}$-pointwise convergence have been studied in this space.

\section{\bf{Preliminaries}}
Now we recall some basic definitions and notations.  If $\mathcal{I}$ is non trivial proper ideal in a non empty set $X$ then the family of sets $\mathcal{F(I)} = \{A\subset X : X\setminus A\in\mathcal{I}\}$ is a filter on $X$ which is called the filter associated with the ideal $\mathcal{I}$. Throughout the paper $\mathcal{I}\subset 2^{\mathbb{N}}$ will stand for an admissible ideal.  

A sequence $\{x_n \}_{n\in\mathbb{N}}$ of real numbers is said to be $\mathcal{I}$-convergent to $x\in\mathbb{R}$ if for each $\varepsilon > 0$ the set $A(\varepsilon)=\{n \in\mathbb{N}: | {x_n - x}| \geq \varepsilon \}\in\mathcal{I}$.
The sequence $\{x_n \}_{n\in\mathbb{N}}$ of real numbers is said to be $\mathcal{I^*}$-convergent to $x\in\mathbb{R}$ if there exists a set $M=\{m_1 < m_2 < \cdots < m_k < \cdots\}\in\mathcal{F(I)}$ such that $x$ is the limit of the subsequence $\{x_{m_k} \}_{k\in\mathbb{N}}$ \cite{Kostyrko Salat Wilczynski}.

Let $f, f_n$ be real valued functions defined on a non empty set $X$. The sequence $\{f_n \}_{n\in\mathbb{N}}$ is said to be equally convergent (\cite{Csaszar Laczkovich}) to $f$ if there exists a sequence $\{\varepsilon _n \}_{n\in\mathbb{N}}$ of positive reals with $\lim_{n\to \infty }  \varepsilon _n = 0$ such that for every $x\in X$ there is $m = m(x)\in\mathbb{N}$ with $ |f_n (x) - f(x) | < \varepsilon _n$ for  $n\geq m$.
In this case we write $f_n \xrightarrow {e}  f$.

Now we see the key ideas of $\mathcal{I}$-uniform convergent \cite{Balcerzak Dems Komisarski} and $\mathcal{I}$ and $\mathcal{I^*}$-equal convergent \cite{Das Dutta Pal} sequences of real valued functions which will be needed for generalizations into linear $2$-normed spaces.

  A sequence $\{f_n \}_{n\in\mathbb{N}}$ is said to be $\mathcal{I}$-uniformly convergent to $f$ if for each $\varepsilon >0$ there exists a set $B \in\mathcal{I}$ such that for all $n\in B^c$ and for all $x\in X$,  $| f_{n} (x) - f(x) | < \varepsilon $. In this case we write $f_n \xrightarrow {\mathcal{I}-u} f$.
$f$ is called $\mathcal{I}$-equal limit of the sequence $\{f_n \}_{n\in\mathbb{N}}$ if there exists a sequence $\{\varepsilon _n \}_{n\in\mathbb{N}}$ of positive reals with $\mathcal{I}\text{-}\lim_{n\to \infty }  \varepsilon _n = 0$ such that for any $x\in X$, the set $\{n \in\mathbb{N} : |f_{n} (x) - f(x) | \geq \varepsilon _{n} \} \in\mathcal{I}$. In this case we write $f_n \xrightarrow {\mathcal{I}-e}f$.
 The sequence $\{f_n \}_{n\in\mathbb{N}}$ is said to be $\mathcal{I^*}$-equal convergent to $f$ if there exists a set $M=\{m_1 < m_2 < \cdots < m_k \cdots\} \in\mathcal{F(I)}$ such that $f$ is the equal limit of the subsequence $\{f_{m_k} \}_{k\in\mathbb{N} }$. In this case we write $f_n \xrightarrow {\mathcal{I^*}-e}f$.

Now we recall the following two important notions which are basically equivalent to each other (due to Lemma 3.9. and Definition 3.10. in \cite{Macaj Sleziak}).
Let $\mathcal{I}\subset 2^{\mathbb{N}}$ be an admissible ideal. $\mathcal{I}$ is called $P$-ideal if for every sequence of mutually disjoint sets $\{A_{1}, A_{2},\cdots\}$ belonging to $\mathcal{I}$ there exists a sequence  $\{B_{1}, B_{2},\cdots\}$ of sets belonging to $\mathcal{I}$ such that $A_{j} \triangle B_{j}$ is finite for $j \in\mathbb{N}$ and $B= \bigcup _{j \in\mathbb{N}} B_{j} \in\mathcal{I}$. This notion is also called $AP$-condition while in \cite{Macaj Sleziak} it is denoted as $AP(\mathcal{I}, Fin)$. An ideal $\mathcal{I}$ is a $P$-ideal if for any sets $A_{1}, A_{2},\cdots$ belonging to $\mathcal{I}$ there exists a set $A \in\mathcal{I}$ such that $A_{n} \setminus A$ is finite for $n \in\mathbb{N}$.

 Now we state some results from \cite{Kostyrko Macaj Salat} for the sequences of real numbers. 
 
\begin{thm}
  Suppose that $\{x_{n} \}_{n\in\mathbb{N}}$ be a sequence of real numbers and $\mathcal{I}$ is an admissible ideal in $\mathbb{N}$. If  $\mathcal{I^*}\text{-} \lim_{n\to \infty }  x _{n} = \xi $ then  $\mathcal{I}\text{-}\lim_{n\to \infty }  x _{n} = \xi $.
\end{thm}

\begin{thm}\label{thm1}
$\mathcal{I}\text{-}\lim_{n\to \infty }  x _{n} = \xi $ implies $\mathcal{I^*}\text{-} \lim_{n\to \infty }  x _{n} = \xi $ if and only if $\mathcal{I}$ Satisfies the condition $AP$.
\end{thm}

  We will now recall the definition of linear $2$-normed spaces which will play very important role throughout the paper.
   
 \begin{defi}
(\cite{Gahler}) Let $X$ be a real vector space of dimension $d$, where $2\leq d< \infty $. A $2$-norm on $X$ is a function $\norm{. , .} : X\times X \rightarrow \mathbb{R}$ which satisfies the following conditions:\\
(C1) $\norm{x,y}=0$ if and only if $x$ and $y$ are linearly dependent in $X$;\\
(C2) $\norm{x,y} = \norm{y, x}$ for all $x,y$ in $X$;\\
(C3) $\norm{\alpha x,y} =|\alpha |\norm{x,y}$ for all $\alpha$ in $\mathbb{R}$ and for all $x,y$ in $X$;\\
(C4) $\norm{x+y,z} \leq \norm{x,z}+\norm{y,z}$ for all $x,y,z$ in $X$.
 \end{defi}
 
The pair $(X, \norm{. , .})$ is called a linear $2$-normed space. A simple example (\cite{Sarabadan Talebi}) of a linear $2$-normed space is $(\mathbb{R}^2 , \norm{.,.})$ where the equipped $2$-norm is given by $\norm{x,y}=|x_{1}y_{2}-x_{2}y_{1}|, \  x=(x_1 ,x_2),  y=(y_1 , y_2) \in\mathbb{R}^2$.

Let $X$ be a $2$-normed space of dimension $d$, $2\leq d<\infty$. A sequence $\{x_n\}_{n\in \mathbb{N}}$ in  $X$ is said to be convergent (\cite{Arslan Dundar}) to $\xi \in X$ if $\lim_{n \to \infty} \norm{x_n -\xi , z} = 0$,
 for every $z \in X$. In such a case $\xi$ is called limit of $\{x_n\}_{n\in \mathbb{N}}$. The sequence $\{x_n\}_{n\in \mathbb{N}}$ in  $X$ is said to be $\mathcal{I}$-convergent (\cite{Sahiner Gurdal Saltan Gunawan}) to $\xi \in X$ if for each $\varepsilon > 0$ and $z\in X$, the set $A(\varepsilon) = \{n\in \mathbb{N} : \norm{x_n - \xi , z} \geq \varepsilon \}\in \mathcal{I}$. The number $\xi$ is called $\mathcal{I} $- limit of $\{x_n\}_{n\in \mathbb{N}}$.

\section{ \bf {Main Results}}

In this paper we study the concepts of $\mathcal{I}$ and $\mathcal{I^*}$-equal convergence of sequences of functions and investigate relationship between them in linear $2$-normed spaces.
Throughout the paper we propose $X$ as a non empty set and $Y$ as a linear $2$-normed space having dimension $d$ with $2\leq d <\infty$.
\begin{defi}
Let $f , f_n : X\rightarrow Y , n\in\mathbb{N}$. The sequence $\{f_n \}_{n\in\mathbb{N}}$ is said to be equally convergent to $f$ if there exists a sequence $\{\varepsilon _n \}_{n\in\mathbb{N}}$ of positive reals with $\lim _{n\to \infty} \varepsilon _n = 0$ such that for every $x \in X$ there is $m=m(x)\in\mathbb{N}$ with   $\norm{ f_n (x) - f(x) , z } < \varepsilon _n$ for  $n\geq m$ and for every $z \in Y$.
In this case we write $f_n \xrightarrow{e} f$.
\end{defi}

\begin{defi}
Let $f , f_n : X\rightarrow Y , n\in\mathbb{N} $. The  sequence $\{f_n \}_{n\in\mathbb{N}}$ is said to be $\mathcal{I}$-uniformly convergent to $f$ if for any $\varepsilon > 0$ there exists a set  $A\in\mathcal{I}$ such that for all $n\in A^c $ and for all $x \in X , z\in Y$, $\norm{ f_n (x) - f(x) , z } < \varepsilon $.
In this case we write $f_n \xrightarrow{\mathcal{I}-u} f$.
\end{defi}

\begin{defi}
Let $f , f_n : X\rightarrow Y , n\in \mathbb{N}$. Then the the sequence $\{f_n\}_{n\in \mathbb{N}}$ is said to be $\mathcal{I}$-equal convergent to $f$ if there exists a sequence $\{\varepsilon _n\}_{n\in \mathbb{N}}$ of positive reals with $\mathcal{I}\text{-}\lim _{n\to \infty} \varepsilon _n =0$ such that for any $x\in X$ and for any $z\in Y$, the set $\{n\in\mathbb{N}: \norm{f_n (x) - f(x) , z} \geq \varepsilon _n \} \in \mathcal{I}$.
In this case $f$ is called $\mathcal{I}$-equal limit of the sequence $\{f_n\}_{n\in\mathbb{N}}$ and we write $f_n \xrightarrow{\mathcal{I}-e} f$.
\end{defi}
\begin{exmp}
Let $\mathcal{I}$ be a non trivial proper admissible ideal. Let $X=\mathbb{R}^2$ and $Y=\{(a,0): a\in \mathbb{R}\}$. 
Define $f_n(x_1 , x_2 ) =(\frac{1}{n+1}, 0)$ 
and $f(x_1 , x_2) = (0, 0) \ \text{for all} \ (x_1 , x_2) \in \mathbb{R}^2$. Suppose $\varepsilon _n = \frac{1}{n}$. Then $\mathcal{I}\text{-}\lim_{n\to \infty}\varepsilon _n=0$. Here we use the $2$-norm on $\mathbb{R}^2$ by $\norm{x,y}=|x_{1}y_{2}-x_{2}y_{1}|, \  x=(x_1 ,x_2),  y=(y_1 , y_2) \in\mathbb{R}^2$. Now we consider the set $A=\{n\in \mathbb{N}:\norm{f_n(x_1 , x_2)-f(x_1 , x_2),z}\geq \varepsilon _n \}$ for all $z=(y_1, y_2)\in Y$. Then  $A=\{n\in \mathbb{N}:\norm{(\frac{1}{n+1},0)-(0,0), (y_1 , y_2)}\geq \frac{1}{n}\}=\{n\in \mathbb{N}: \frac{y_2}{n+1}\geq \frac{1}{n}\}=\{n\in \mathbb{N}: 0\geq \frac{1}{n}\}=\phi \in \mathcal{I}$, since $y_2=0$. Therefore $f_n \xrightarrow{\mathcal{I}-e} f$. 
\end{exmp}
Now we investigate some arithmetical properties of $\mathcal{I}$-equal convergent sequences of functions.
\begin{thm}
 Let $f , f_n : X\rightarrow Y , n\in\mathbb{N}$. If $f_n \xrightarrow{\mathcal{I}-e} f$  then $f$ is unique. 
\end{thm}

\begin{proof} If possible let $f$ and $g$ be two distinct $\mathcal{I}$-equal limit of $\{f_n\}_{n\in \mathbb{N}}$. Then there are two sequences $\{\varepsilon _n \}_{n\in\mathbb{N}}$ and $\{\gamma _n \}_{n\in\mathbb{N}}$ of positive reals with $\mathcal{I}\text{-}\lim _{n\to \infty} \varepsilon _n =0$ and $\mathcal{I}\text{-}\lim _{n\to \infty} \gamma _n =0$ and for any $x\in X$ and for any $z\in Y$, the sets $K_1 = \{n\in\mathbb{N}: \norm{f_n(x) -f(x) , z }\geq \varepsilon _n \}$,
  $K_2 = \{n\in\mathbb{N}: \norm{f_n(x) -g(x) , z } \geq \gamma _n \} \in\mathcal{I}$. 
Therefore   $ K_1 ^c = \{n\in\mathbb{N}: \norm{f_n(x) -f(x) , z } < \varepsilon _n \}$,
 $K_2 ^c = \{n\in\mathbb{N}: \norm{f_n(x) -g(x) , z } < \gamma _n \} \in\mathcal{F(I)}$.
Let $z\in Y$  be linearly independent with $f(x) - g(x)$. Put $\varepsilon = \frac{1}{2}\norm{f(x) - g(x) , z} > 0 $. As $\mathcal{I}\text{-}\lim _{n\to \infty} \varepsilon _n =0$ and $\mathcal{I}\text{-}\lim _{n\to \infty} \gamma _n =0$, the sets  $K_3 ^c = \{n\in\mathbb{N} : \varepsilon _n < \varepsilon \}, \ K_4 ^c = \{n\in\mathbb{N} : \gamma _n < \varepsilon \} \in\mathcal{F(I)}$. 
As $\phi  \notin\mathcal{F(I)},\  K_1 ^c \cap K_2 ^c \cap K_3 ^c \cap K_4 ^c \neq \phi $. Then there exists $m\in\mathbb{N}$ such that $m\in K_1 ^c \cap K_2 ^c \cap K_3 ^c \cap K_4 ^c  $.
Then $\norm{f_m(x) -f(x) , z } < \varepsilon _m ,\  \norm{f_m(x) -g(x) , z } < \gamma _m ,\ \varepsilon _m < \varepsilon $ and $\gamma _m < \varepsilon $. 
Now 
$   \norm{f(x) - g(x), z} = \norm{f(x) - f_m (x) + f_m(x) - g(x) , z }   \leq \norm{f_m(x) -f(x) , z } +  \norm{f_m(x) -g(x) , z} 
 <\varepsilon _m +\gamma _m 
 < \varepsilon + \varepsilon \\
 = \frac{1}{2}\norm{f(x) - g(x) , z} + \frac{1}{2}\norm{f(x) - g(x), z} 
 = \norm{f(x) - g(x) , z},\ \text{which is absurd}$.
 Hence $\mathcal{I}$-equal limit $f$ of the sequence $\{f_n \}_{n\in\mathbb{N}}$ must be unique if it exists.
 \end{proof}
\begin{thm}
Let $f , f_n : X\rightarrow Y $ and $g , g_n : X\rightarrow Y$, $n\in\mathbb{N}$. If $f_n \xrightarrow{\mathcal{I}-e} f$ and $g_n \xrightarrow{\mathcal{I}-e} g$, $f_n + g_n \xrightarrow{\mathcal{I}-e} f + g$.
\end{thm}

\begin{proof} Since $f_n \xrightarrow{\mathcal{I}-e} f$ and $g_n \xrightarrow{\mathcal{I}-e} g$, there exist  sequences $\{\xi _n \}_{n\in\mathbb{N}}$ and $\{\rho _n \}_{n\in\mathbb{N}}$ of positive reals with $\mathcal{I}\text{-}\lim _{n\to \infty} \xi _n =0$ and $\mathcal{I}\text{-}\lim _{n\to \infty} \rho _n =0$ such that 
 for $x\in X $ and $z\in Y$, we have 
 $A_1 = \{n\in\mathbb{N}: \norm{f_n(x) -f(x) , z } \geq \xi _n \},\ A_2 = \{n\in\mathbb{N}: \norm{g_n(x) -g(x) , z } \geq \rho _n \} \in\mathcal{I}$.
So  $A_1 ^c = \{n\in\mathbb{N}: \norm{f_n(x) -f(x) , z} < \xi _n \}, \ 
   A_2 ^c = \{n\in\mathbb{N}: \norm{g_n(x) -g(x) , z} < \rho _n \} \in\mathcal{F(I)}$.
 As $\phi \notin\mathcal{F(I)} , A_1 ^c \cap A_2 ^c  \neq \phi $. 
Let $n\in A_1 ^c \cap A_2 ^c $ and consider the set 
$ A_3 ^c = \{n\in\mathbb{N}: \norm{f_n(x)+ g_n (x) -\{f(x)+g(x)\} , z} < \xi _n +\rho _n \}$.
As, $ \norm{f_n(x)+ g_n (x) -\{f(x)+g(x)\} , z} 
 \leq \norm{f_n(x) -f(x) , z} + \norm{g_n(x) -g(x) , z}
 <  \xi _n +  \rho _n$. 
Therefore $n\in A_3 ^c$ i.e. $ A_1 ^c \cap A_2 ^c \subset A_3 ^c  $. So $A_3 \subset A_1 \cup A_2$.
Since $A_1 \cup A_2 \in\mathcal{I}$, $A_3 \in\mathcal{I}$.
i.e. $\{ n \in\mathbb{N} : \norm{f_n(x)+ g_n (x) -\{f(x)+g(x)\} , z} \geq   \xi _n +\rho _n \} \in\mathcal{I}$. As $\mathcal{I}\text{-}\lim _{n\to \infty} \xi _n + \rho _n =0$,
 $f_n + g_n \xrightarrow{\mathcal{I}-e} f + g$.
This proves the theorem.
\end{proof}

\begin{thm}
Let $f , f_n : X\rightarrow Y , n\in\mathbb{N}$. Let $a(\neq 0)\in\mathbb{R}$. If $f_n \xrightarrow{\mathcal{I}-e} f$,  $af_n \xrightarrow{\mathcal{I}-e} af$.
\end{thm}

\begin{proof} Since  $f_n \xrightarrow{\mathcal{I}-e} f$, there is a sequence $ \{\beta _n \}_{n\in\mathbb{N}} $ of positive reals with  $\mathcal{I}\text{-}\lim _{n\to \infty} \beta _n =0$ such that for $x\in X , z\in Y$, the set $B_1 = \{n\in\mathbb{N}: \norm{f_n(x) -f(x) , z} \geq \frac{\beta _n}{|a|} \} \in\mathcal{I}$. 
Put $ B_2 = \{n\in\mathbb{N}: \norm{af_n(x) -af(x) , z } \geq \beta _n \}$.
As, $ \norm{af_n(x) -af(x) , z} \geq \beta _n   
\Rightarrow \norm{f_n(x) -f(x) , z} \geq \frac{\beta _n}{|a|}$.
Therefore $B_2 \subset B_1$. So $B_2 \in I$.
This proves the result.
\end{proof}
In \cite{Das Dutta Pal} it has been proved for real valued functions that $\mathcal{I}$-uniform convergence implies $\mathcal{I}$-equal convergence . Now we investigate it in linear $2$-normed spaces which will be needed in the sequel. 
First we give  an important lemma which has been stated as remark in \cite{Sarabadan Talebi}. Here we give a detailed proof for the sake of completeness.

 \begin{lem}(cf.\cite{Sarabadan Talebi})\label{lem1}
 Let $f , f_n : X\rightarrow Y , n\in\mathbb{N}$. Then $ \{f _n \}_{n\in\mathbb{N}} $ is $\mathcal{I}$-uniformly convergent to $f$ if and only if $\{\sup _{x\in X} \norm{f_n(x) -f(x) , z } \}_{n\in\mathbb{N}}$ is $\mathcal{I}$-convergent to zero for all $z\in Y$.
 \end{lem}

\begin{proof} 
First we assume that  $ \{f _n \}_{n\in\mathbb{N}} $ is $\mathcal{I}$-uniformly convergent to $f$. Then for any $\varepsilon  > 0$ there exists $M \in\mathcal{I}$ such that for all $n\in M^c$ and for $x\in X$, $z\in Y$, $\norm{f_n(x) -f(x) , z } < \frac{\varepsilon}{2}$.
This implies $\sup _{x\in X} \norm{f_n(x) -f(x) , z } \leq \frac{\varepsilon}{2} < \varepsilon$. 
So the set $\{n\in\mathbb{N} :| \sup _{x\in X} \norm{f_n(x) -f(x) , z } - 0| \geq \varepsilon \} \subset M\in \mathcal{I},  \text{for all} z\in Y$.
Therefore  $\{\sup _{x\in X} \norm{f_n(x) -f(x) , z} \}_{n\in\mathbb{N}}$ is $\mathcal{I}$-convergent to zero for all $z\in Y$.

Conversely assume that  $\{\sup _{x\in X} \norm{f_n(x) -f(x) , z } \}_{n\in \mathbb{N}}$ is $\mathcal{I}$-convergent to zero for all $z\in Y$.
So for any $\varepsilon > 0$ the set $\{n\in\mathbb{N} :| \sup _{x\in X} \norm{f_n(x) -f(x) , z } - 0| \geq \varepsilon \} \in \mathcal{I}$. So there exists $A\in\mathcal{I}$ such that for all $n\in A^c$ and for any $z\in Y$, $\sup _{x\in X} \norm{f_n(x) -f(x) , z} < \varepsilon$.
Therefore $\norm{f_n(x) -f(x) , z } \leq \sup _{x\in X} \norm{f_n(x) -f(x) , z} < \varepsilon$ for all $n\in A^c$.
Hence   $ \{f _n \}_{n\in\mathbb{N}} $ is $\mathcal{I}$-uniformly convergent to $f$ in $Y$. 
\end{proof}

\begin{thm}\label{thm2}
 Let $f , f_n : X\rightarrow Y , n\in\mathbb{N}$. $f_n \xrightarrow{\mathcal{I}-u} f$  implies  $f_n \xrightarrow{\mathcal{I}-e} f$.
\end{thm}

\begin{proof} Since the sequence $ \{f _n \}_{n\in\mathbb{N}}$ is $\mathcal{I}$-uniformly convergent to $f$ in $Y$, due to the Lemma \ref{lem1} the sequence $\{u_n \}_{n\in\mathbb{N}}$ is $\mathcal{I}$-convergent to zero where $u_n = \sup _{x\in X} \norm{f_n(x) -f(x) , z} $, for all $z\in Y$.
Let $\varepsilon > 0 $ be given. Then the set $B = \{n\in\mathbb{N} : u_n \geq \varepsilon \} \in\mathcal{I}$. 
Define \\
$ \xi _n = \begin{cases}
\frac{1}{n} , & \text{if}  \ n\in B \\
u_n + \frac{1}{n} , & \text{if} \  n\notin B 
\end{cases}$.
We show $\{\xi _n \}_{n\in\mathbb{N}}$ is $\mathcal{I}$-convergent to zero.
For, let $\varepsilon _1 > 0$, we have 
$\{n : \xi _n \geq \varepsilon _1 \} 
= \{n \in B: \xi _n \geq \varepsilon _1 \} \cup  \{n \in B^c: \xi _n \geq \varepsilon _1 \}  
=\{n : \frac{1}{n} \geq \varepsilon _1 \} \cup \{n : u_n + \frac{1}{n} \geq \varepsilon _1 \} 
= M_1 \cup M_2$.
Clearly $M_1$ is finite. 
If $n\in M_2 $ then $n\in B^c$.
So $u_n < \varepsilon$.
Now $u_n + \frac{1}{n} \geq \varepsilon _1 $ if $\frac{1}{n} \geq \varepsilon _1 - u_n$ 
i.e. if   $\frac{1}{n} \geq \varepsilon _1 - \varepsilon$ which is for finite number values of $n$.
Therefore $M_2$ is finite.
 As $\mathcal{I}$ is admissible, $M_1 \cup M_2 \in\mathcal{I}$.
Hence $\mathcal{I}\text{-}\lim _{n \to \infty } \xi _n =0$. 
Now, for all $z\in Y$, we have 
$  \norm{f_n(x) -f(x) , z } 
 \leq  \sup _{x\in X} \norm{f_n(x) -f(x) , z } 
 <  \sup _{x\in X} \norm{f_n(x) -f(x) , z } + \frac{1}{n}
 =  \ u_n + \frac{1}{n} 
 =  \ \xi _n \ \text{if} \ n\in B^c \  \text{where} \ B\in \mathcal{I}$.
Therefore  $\{n\in\mathbb{N} : \norm{f_n(x) -f(x) , z } \geq \xi _ n \} \in\mathcal{I}$. As $\mathcal{I}\text{-}\lim _{n\to \infty } \xi _n =0$, $f_n \xrightarrow{I-e} f$.
Hence the theorem follows.
\end{proof}
Now we intend to proceed with the notion of $I^*$-equal convergence in linear $2$-normed spaces.
\begin{defi}
Let $f , f_n : X\rightarrow Y , n\in\mathbb{N}$. The sequence $\{f_n\}_{n\in\mathbb{N}}$ is said to be $\mathcal{I^*}$-equal convergent to $f$ if there exists a set $M = \{ m_1 < m_2 <  \cdots < m_k \cdots  \} \in\mathcal{F(I)}$ and a sequence $\{\varepsilon _k \}_{k\in M}$ of positive reals with $\lim _{k\to \infty} \varepsilon _k =0 $ such that for every $x\in X$, there is a number $p\in\mathbb{N}$ and for every $z\in Y$,   $\norm{f_{m_k}(x) -f(x) , z } < \varepsilon _k $ for all $k \geq p$.
In this case we write $f_n \xrightarrow{\mathcal{I^*}-e} f$.
\end{defi}

 We proceed to investigate the relationship between $\mathcal{I}$-equal and $\mathcal{I^*}$-equal convergence in linear $2$-normed spaces. 
\begin{thm}\label{thm5}
Let $f , f_n : X\rightarrow Y , n\in\mathbb{N}$. If  $f_n \xrightarrow{\mathcal{I^*}-e}f $ then $f_n \xrightarrow{\mathcal{I}-e} f$. 
\end{thm}

\begin{proof} We assume $f_n \xrightarrow{\mathcal{I^*}-e}f$. Then there exist a set  $M = \{ m_1 < m_2 <  \cdots < m_k \cdots  \} \in \mathcal{F(I)}$ and a sequence $\{\varepsilon _k \}_{k\in M}$ of positive reals with $\lim _{k\to \infty} \varepsilon _k =0 $ such that for every $x\in X$, there is a number $p\in\mathbb{N}$ and for every $z\in Y$,   $\norm{f_{m_k}(x) -f(x) , z } < \varepsilon _k $ for  $k > p$.\\
Then clearly  $\norm{f_n (x) -f(x) , z} \geq  \varepsilon _n $ holds for $n\in (N\setminus M) \cup \{ m_1 , m_2 ,  \cdots , m_p   \} $. This implies $\{n : \norm{f_n (x) -f(x) , z } \geq  \varepsilon _n \} \subset  (\mathbb{N}\setminus M) \cup \{ m_1 , m_2 ,  \cdots , m_p \}$.
Since $\mathcal{I}$ is admissible, $\{n : \norm{f_n (x) -f(x) , z } \geq  \varepsilon _n \} \in \mathcal{I}$. 
Hence $f_n \xrightarrow{\mathcal{I}-e} f$.
\end{proof}

\begin{rem}
The converse of the above theorem may not hold in general as shown by the following example.
\end{rem}

\begin{exmp}
Consider a decomposition $\mathbb{N} = \bigcup _{i=1} ^{\infty} 2^{j-1} (2s-1) $.
Let $\mathcal{I}$ be the class of all subsets of $\mathbb{N}$ which intersects finite umber of $D_i ^{'}s$. Then $\mathcal{I}$ is an admissible ideal.
Let $f , f_n : X\rightarrow Y , n\in\mathbb{N}$ such that $\{f_n \}_{n\in  N}$ is uniformly convergent to $f$. Then for each $\varepsilon > 0$ there exists $p\in\mathbb{N}$ such that for all $x\in X , z\in Y $, $\norm{f_n (x) -f(x) , z } < \varepsilon$ for all $n> p$.
Define a sequence $\{g_n \}_{n\in\mathbb{N}}$ by $g_n = f_j \ \text{if} \ n\in D_j$.  Then for all $x\in X, z\in Y$ the set $\{ n \in\mathbb{N} :\norm{g_n (x) -f(x) , z} \geq  \varepsilon  \} \subset D_1 \cup D_2 \cup \cdots \cup D_p $.  Therefore $\{ n \in\mathbb{N} :\norm{g_n (x) -f(x) , z } \geq  \varepsilon  \} \in\mathcal{I}$. Hence  $g_n \xrightarrow{\mathcal{I}-u} f$.
By the Theorem \ref{thm2},  $g_n \xrightarrow{\mathcal{I}-e} f$.\\
Now we shall show that $\{g_n\}_{n\in\mathbb{N}}$ is not $\mathcal{I^*}$-equal convergent in $Y$. 
If possible let  $g_n \xrightarrow{\mathcal{I^*}-e} f$.
Now, by definition, if $H\in\mathcal{I}$, then there is a $p\in \mathbb{N}$ such that $H\subset D_1 \cup D_2 \cup \cdots \cup D_p $. Then $D_{P+1} \subset \mathbb{N}\setminus H$ and so we have $g_{m_k} = f_{p+1}$ for infinitely many of $ k^{'} s$.
Let $z\in Y$ be  linearly independent with $f_{p+1} - f(x)$. 
Now we have $\lim _{n\to \infty} \norm{g_{m_k} (x) - f(x) , z} 
= \norm{f_{p+1} (x) - f(x) , z } 
\neq 0 $.
Which shows that  $\{g_n \}_{n\in \mathbb{N}}$ is not $\mathcal{I^*}$-equal convergent in $Y$.
\end{exmp}
Now we see, if $\mathcal{I}$ satisfies the condition $AP$ and $X$ is countable then the converse of the Theorem \ref{thm5} also holds. In the next theorem we  investigate  whether the two concepts  $f_n \xrightarrow{\mathcal{I}-e} f$  and  $f_n \xrightarrow{\mathcal{I^*}-e} f$  coincide in linear $2$-normed spaces when  $\mathcal{I}$ is a $P$-ideal.

\begin{thm}
Let $f , f_n : X\rightarrow Y , n\in\mathbb{N}$ and $X$ is a countable set. Then  $f_n \xrightarrow{\mathcal{I}-e} f$ implies $f_n \xrightarrow{\mathcal{I^*}-e} f$  whenever $\mathcal{I}$ is a $P$-ideal.
\end{thm}

\begin{proof}
 From the given condition there exists a sequence $\{\sigma _n \}_{n\in\mathbb{N}}$ of positive reals with \\ $\mathcal{I}\text{-}\lim _{n\to \infty } \sigma _n = 0$ and for every $z\in Y$ and for each $x\in X$, there is a set $B = B(x) \in\mathcal{F(I)}$, $\norm{f_n (x) -f(x) , z} < \sigma _n$  for all $n\in B$. Now by Theorem \ref{thm1}, $\mathcal{I^*}\text{-}\lim _{n\to \infty } \sigma _n = 0$. So we will get a set $H\in\mathcal{F(I)}$ for which $\{\sigma _n \}_{n\in H}$ is convergent to zero. Since  $X$ is countable, let us enumerate the elements of $X$ as follows: $X=\{x_1 , x_2 , \cdots \}$ and so for each element $x_i$ of it and for every $z\in Y$, there is a set $B_i = B(x_i ) \in\mathcal{F(I)}$, we have $\norm{f_n (x_i) -f(x_i) , z} < \sigma _n$ for all $n\in B_i$. $\mathcal{\mathcal{I}}$-being a $P$-ideal, there is a set $A\in\mathcal{F(I)}$ such that $A\setminus B_i$ is finite for all $i$. So for all $z\in Y$ and for all $n\in A\cap H$ except for finite number of values, we have $\norm{f_n (x) -f(x) , z } < \sigma _n $. 
Therefore  $f_n \xrightarrow{\mathcal{I^*}-e} f$.  
Hence the theorem follows. 
\end{proof}
\begin{thm}
Let $f , f_n : X\rightarrow Y , n\in\mathbb{N}$. Suppose that  $f_n \xrightarrow{\mathcal{I}-e} f$ implies $f_n \xrightarrow{\mathcal{I^*}-e}f$. Then $\mathcal{I}$ satisfies the condition $AP$.
\end{thm}

\begin{proof} Let $f , f_n : X\rightarrow Y , n\in\mathbb{N}$ such that $\{f_n \}_{n\in\mathbb{N}}$ is uniformly convergent to $f$. Then for each $\varepsilon > 0$ there exists $p\in\mathbb{N}$ such that for all $x\in X , z\in Y$ , $\norm{f_n (x) -f(x) , z} < \varepsilon $ for all $n > p$.
Suppose $\{ M_1 , M_2 , \cdots \}$ be a class of mutually disjoint non empty sets from $\mathcal{I}$. 
Define a sequence $\{h_n\}_{n\in\mathbb{N}}$ by 
$h_n = \begin{cases}
f_j , & \text{if} \ n\in M_j \\
f , & \text{if}  \ n\in\mathbb{N}\setminus \bigcup _{j} M_j 
\end{cases}$.
First of all we shall show that $h_n \xrightarrow{\mathcal{I}-u} f$.
Let $\varepsilon > 0 $ be given. Observe that the set  $M = M_1 \cup M_2 \cup \ldots \cup M_p \in\mathcal{I}$ and for all $x\in X , z\in Y$, we have
 $\norm{h_n (x) -f(x) , z } < \varepsilon$ for all $n\in M^c$. 
  i.e.  $\{ n\in\mathbb{N} : \norm{h_n (x) -f(x) , z } \geq  \varepsilon \} \subset M_1 \cup M_2 \cup \cdots \cup M_p\in\mathcal{I}$.  
Therefore $h_n \xrightarrow{\mathcal{I}-u} f$. 
By the Theorem \ref{thm2} we have   $h_n \xrightarrow{\mathcal{I}-e} f$.
So by the given condition $h_n \xrightarrow{\mathcal{I^*}-e} f$.
Therefore there is a set $B\in\mathcal{I}$ such that 
\begin{equation}\label{eqn1}
H =\mathbb{N}\setminus B = \{ a_1 < a_2 < \cdots < a_k < \cdots \}\in\mathcal{F(I)}  \   \text{and} \ h_{a_k} \xrightarrow{e}f.    
\end{equation} 
Put $B_j = M_j \cap B \ (j = 1, 2, \cdots ) $. So $\{B_1 , B_2 ,\cdots \}$ is a class of sets belonging to $\mathcal{I}$. Now $\bigcup _{j = 1} ^{\infty} B_j = \bigcup _{j=1} ^{\infty} (M_j \cap B)= (B \cap \{\bigcup _{j=1} ^{\infty} M_j \} \subset B$.
Since $B \in\mathcal{I}$ it follows $\bigcup _{j = 1} ^{\infty} B_j \in\mathcal{I
}$.
Now from the equation \ref{eqn1} we see that the set $M_j$ has a finite number of elements common with the set $\mathbb{N}\setminus B$. So $M_j \triangle B_j \subset M_j \cap (\mathbb{N}\setminus B) $. Therefore $M_j \triangle B_j$ is finite. Therefore $\mathcal{I}$ satisfies the condition $AP$.
\end{proof}

\textbf{Acknowledgement.}
The second author is grateful to The Council of  Scientific  and Industrial Research, HRDG, India, for the grant of Junior Research Fellowship during the preperation of this paper.

\end{document}